\documentclass[reqno,10pt,centertags]{amsart}
\usepackage{amsmath,amsthm,amscd,amssymb,latexsym,upref,stmaryrd}
\usepackage{amsfonts,mathrsfs}
\usepackage{color}
\usepackage{hyperref} 
\newcommand*{\mailto}[1]{\href{mailto:#1}{\nolinkurl{#1}}}

\newcommand{\bbC}{{\mathbb{C}}}
\newcommand{\bbD}{{\mathbb{D}}}

\newcommand{\bbN}{{\mathbb{N}}}

\newcommand{\bbR}{{\mathbb{R}}}

\newcommand{\cB}{{\mathcal B}}

\newcommand{\cE}{{\mathcal E}}

\newcommand{\cH}{{\mathcal H}}
\newcommand{\cI}{{\mathcal I}}

\newcommand{\cK}{{\mathcal K}}

\newcommand{\cU}{{\mathcal U}}

\DeclareMathOperator{\rank}{rank}
\DeclareMathOperator{\ran}{ran}
\DeclareMathOperator{\dom}{dom}
\DeclareMathOperator{\linspan}{lin.span}

\DeclareMathOperator*{\slim}{s-lim}

\renewcommand{\Im}{\text{\rm Im}}

\newcommand{\beq}{\begin{equation}}
\newcommand{\enq}{\end{equation}}

\newcommand{\no}{\notag}
\newcommand{\lb}{\label}

\newcommand{\ol}{\overline}

\newcommand{\wti}{\widetilde}

\newcommand{\bi}{\bibitem}

\let\geq\geqslant
\let\leq\leqslant

\makeatletter
\def\theequation{\@arabic\c@equation}

\allowdisplaybreaks 
\numberwithin{equation}{section}

\newtheorem{theorem}{Theorem}[section]

\newtheorem{lemma}[theorem]{Lemma}

\newtheorem{conjecture}[theorem]{Conjecture}

\newtheorem{hypothesis}[theorem]{Hypothesis}
\newtheorem{example}[theorem]{Example}
\theoremstyle{remark}
\newtheorem{remark}[theorem]{Remark}

\begin{document}

\title[A Problem in Perturbation Theory]{On a Problem in Eigenvalue Perturbation Theory}

\author[F.\ Gesztesy]{Fritz Gesztesy}
\address{Department of Mathematics,
University of Missouri, Columbia, MO 65211, USA}
\email{\mailto{gesztesyf@missouri.edu}}
\urladdr{\url{http://www.math.missouri.edu/personnel/faculty/gesztesyf.html}}

\author[S.\ N.\ Naboko]{Sergey N.\ Naboko}
\address{Department of Mathematical Physics, St.~Petersburg  State University, Ulianovskaia 1, NIIF, St.~Peterhof, St.~Petersburg,
Russian Federation, 198504}
\email{\mailto{sergey.naboko@gmail.com}}

\author[R.\ Nichols]{Roger Nichols}
\address{Mathematics Department, The University of Tennessee at Chattanooga, 415 EMCS Building, Dept. 6956, 615 McCallie Ave, Chattanooga, TN 37403, USA}
\email{\mailto{Roger-Nichols@utc.edu}}
\urladdr{\url{http://www.utc.edu/faculty/roger-nichols/index.php}}

\date{\today}
\thanks{S.N. is supported by grants NCN t 2013/09/BST1/04319, RFBR 12-01-00215-a, and Marie Curie grant PIIF-GA-2011-299919.}
\thanks{R.N. gratefully acknowledges support from an AMS--Simons Travel Grant.} 
\thanks{{\it J. Math. Anal. Appl.} (to appear).} 

\begin{abstract} 
We consider additive perturbations of the type $H_t=H_0+t V$, $t\in [0,1]$, where $H_0$ and $V$ are self-adjoint operators in a separable Hilbert space $\cH$ and $V$ is bounded.  In addition, we assume that the range of $V$ is a generating (i.e., cyclic) subspace for $H_0$.  If $\lambda_0$ is an eigenvalue of $H_0$, then under the additional assumption that $V$ is nonnegative, the Lebesgue measure of the set of all $t\in [0,1]$ for which $\lambda_0$ is an eigenvalue of $H_t$ is known to be zero.  We recall this result with its proof and show by explicit counterexample that the nonnegativity assumption $V\geq 0$ cannot be removed.  
\end{abstract}

\maketitle


\section{Introduction}  \lb{s1}

The focus of this paper is a natural conjecture concerning the eigenvalues of a one-parameter family of self-adjoint perturbations $H_t$, $t\in [0,1]$, of the form
\begin{equation}
H_t = H_0 + t V,\quad \dom(H_t) = \dom(H_0), \quad t\in [0,1],
\end{equation}
where $H_0$ is a (possibly unbounded) self-adjoint operator in a separable Hilbert space $\cH$ and $V$ is a bounded self-adjoint operator in $\cH$.  If $\lambda_0\in \sigma_p(H_0)$ (with $\sigma_p(T)$ denoting the point spectrum, that is, the set of eigenvalues, of a densely defined closed operator $T$ in $\cH$), a natural question is, ``For which values of $t\in [0,1]$ is $\lambda_0$ also an eigenvalue of $H_t$?''  In the general context, a direct answer to this question is likely out of reach.  This is not a problem, however, as many of the applications where this question naturally arises (e.g., in studying the eigenvalues of Anderson-type models, see \cite{NNS13}) do not require one to explicitly determine the set of $t$.  Instead, one only needs the set of such $t$ to be ``small'' in a certain sense.

Without further assumptions on $H_0$ or $V$, it is easy to construct explicit examples for which $H_t$, $t\in [0,1]$, share a common eigenvalue.  For example, choose $\phi\in \cH\backslash \{0\}$ and set $H_0 = V = (\phi, \,\cdot\, )_{{}_{\cH}}\phi$, with $(\,\cdot\, , \, \cdot \,)_{{}_{\cH}}$ denoting the inner product in $\cH$.  Obviously, $0\in \sigma_p(H_0 + t V)$ for all $t\in [0,1]$.  The problem with this example is of course that the range of $V$ is too small.  

To exclude such examples, in the general setting we shall henceforth assume that the range of $V$ is a generating subspace for $H_0$, that is, 
\begin{equation}
\cH =\ol{\linspan \, \{(H_0-zI_{\cH})^{-1} V e_n\in \cH \, \vert\, n \in \cI, \, z\in \bbC \backslash \bbR\}},
\end{equation}
for a complete orthonormal system $\{e_n\}_{n\in\cI}$,  with $\cI\subseteq\bbN$ an appropriate index set.  Then, by appealing to a special form of spectral averaging \cite[Corollary 4.2]{CH94}, it was shown in \cite[Lemma 4]{NNS13} that the Lebesgue measure of the set of $t\in [0,1]$ for which $\lambda_0$ is an eigenvalue of $H_t$ is equal to zero, that is,
\begin{equation}\lb{1.3}
|\{t\in [0,1]\, | \, \lambda_0\in \sigma_p(H_t) \}| = 0,
\end{equation}
under the additional assumption that $V$ is {\it nonnegative}, that is, $V\geq 0$.  Here, by some abuse of notation, we abbreviate the Lebesgue measure on $\bbR$ by $|\cdot|$. 

Actually, \eqref{1.3} is only a special case of the result obtained in \cite[Lemma 4]{NNS13} as, more generally, the authors of \cite{NNS13} show that for any measurable set $M$ of Lebesgue measure zero, the set of all $t$ in $[0,1]$ for which $H_t$ has an eigenvalue in $M$ has Lebesgue measure equal to zero, which is to say
\begin{equation}\lb{1.4}
|\{t\in [0,1]\, |\, \sigma_p(H_t)\cap M \neq \emptyset \}| = 0.
\end{equation}
Obviously, \eqref{1.3} is obtained by choosing $M=\{\lambda_0\}$ in \eqref{1.4}.

Upon seeing the measure zero result in \eqref{1.3}, it is natural to inquire about the extent to which the assumption that $V\geq 0$ is necessary.  This leads to the question: {\it Can one remove the assumption $V\geq 0$ and still retain the conclusion in \eqref{1.3}?}

A careful examination of the proofs to \cite[Lemma 4]{NNS13} and its key ingredient \cite[Corollary 4.2]{CH94} reveals that nonnegativity of $V$ is crucial for both, so dispensing with the assumption $V\geq 0$ (if possible) would require entirely new ideas.  However, as it turns out, the nonnegativity assumption on $V$ is absolutely crucial.  We show that without $V\geq 0$, \eqref{1.3} breaks down in dramatic fashion for in this case one can actually construct a counterexample in the finite dimensional Hilbert space $\cH = \bbC^3$ (cf.\ Example \ref{e2.9}). 

Next, we briefly summarize the contents of this paper.  In Section \ref{s2} we present Conjecture \ref{c2.2}, the natural conjecture that \eqref{1.3} continues to hold with the assumption $V\geq 0$ removed, and in Lemma \ref{l2.3}, we consider perturbations of the form $H_t = H_0 + tV$, $t\in [0,1]$, where $H_0$ and $V$ are self-adjoint and $V$ is compact.  By applying the Birman--Schwinger principle, it is shown that there exist at most finitely many $t\in [0,1]$ for which a given point $E_0\in \bbR\cap \rho(H_0)$ belongs to $\sigma_p(H_t)$.  (The assumption that $E_0\in \rho(H_0)$, the resolvent set of $H_0$, which is necessary in order to apply the Birman--Schwinger principle to $H_0 + tV$, means that this result is fundamentally different from results like \eqref{1.3}, where $\lambda_0\in \sigma_p(H_0)$ is assumed.) In Lemma \ref{l2.5}, we recall the result of \cite[Lemma 4]{NNS13} tailored to the present context, and we provide its proof for completeness.  Following Lemma \ref{l2.5}, we provide a general discussion which relates the eigenvalue problem 
\begin{equation}
H_t\psi = \lambda \psi,\quad \psi\in \cH,\; \lambda\in \bbR, \; t\in [0,1],
\end{equation}
to a linear pencil eigenvalue problem with respect to the $t$ parameter.  Example \ref{e2.7} provides an example where spectra are computed by looking at the corresponding linear pencil.  In our main Example \ref{e2.9}, we put Conjecture \ref{c2.2} to rest, showing by counterexample, that one cannot remove the assumption $V\geq 0$ from Lemma \ref{l2.5} and retain \eqref{1.3}.

Finally, we briefly summarize some of the notation used in this paper: Let $\cH$ be a separable
complex Hilbert space, $(\cdot,\cdot)_{\cH}$ the scalar product in
$\cH$ (linear in the second entry), and $I_{\cH}$ the identity operator in $\cH$.
Next, let $T$ be a linear operator mapping (a subspace of) a
Banach space into another, with $\dom(T)$, $\ran(T)$, and $\ker(T)$ denoting the
domain, range, and kernel (i.e., null space) of $T$.  If $T$ is densely defined, then $T^*$ denotes the Hilbert space adjoint of $T$.  The closure of a closable operator $S$ is denoted by $\ol S$. If $A$ is self-adjoint in $\cH$, we denote
$A_{\pm} = [|A| \pm A]/2$, employing the spectral theorem for $A$. 

The spectrum, discrete spectrum, point spectrum, continuous spectrum, residual spectrum, and resolvent set of a closed linear operator in $\cH$ 
will be denoted by $\sigma(\cdot)$, $\sigma_{d}(\cdot)$, $\sigma_{p}(\cdot)$, 
$\sigma_{c}(\cdot)$, $\sigma_{r}(\cdot)$, and $\rho(\cdot)$, respectively 
(cf., e.g., \cite[p.\ 451--452]{Ku01}). 

The Banach spaces of bounded and compact linear operators in $\cH$ are
denoted by $\cB(\cH)$ and $\cB_\infty(\cH)$, respectively. 

We denote by $E_A(\cdot)$ the family of strongly right-continuous spectral projections of a self-adjoint operator $A$ in $\cH$ (in particular, $E_A(\lambda) = E_A((-\infty,\lambda])$, 
$\lambda \in \bbR$).

\section{On an Eigenvalue Perturbation Problem}  \lb{s2}

Denoting by $|M|$ the Lebesgue measure of a measurable subset $M$ of $\bbR$, and supposing that $\{e_n\}_{n\in\cI}$ (with $\cI\subseteq\bbN$ an appropriate index set) represents a complete orthonormal system in $\cH$, we will always  assume that  we 
assume the following:

\begin{hypothesis} \lb{h2.1} 
Suppose the range of $V$ is a generating $($i.e., cyclic\,$)$ subspace for $H_0$, that is,
\begin{equation}
\cH =\ol{\linspan \, \{(H_0-zI_{\cH})^{-1} V e_n\in \cH \, \vert\, n \in \cI, 
\, z\in \bbC \backslash \bbR\}} 
\end{equation}
$($equivalently, $\cH=\ol{\linspan \, \{E_{H_0}(\lambda) V e_n\in \cH \, \vert \, n \in \cI, \, \lambda \in \bbR\}}$$)$.
\end{hypothesis}

The principal purpose of this paper is to {\bf disprove} the following somewhat naturally sounding conjecture:

\begin{conjecture} \lb{c2.2} 
Assume Hypothesis \ref{h2.1}, let $H_0$ be self-adjoint in $\cH$, and suppose 
$V = V^* \in \cB(\cH)$. Consider 
\begin{equation}
H_t = H_0 + t V, \quad \dom(H_t) = \dom(H_0), \quad t \in [0,1], 
\end{equation} 
and suppose that $\lambda_0 \in \sigma_{p}(H_0)$. Then 
\begin{equation}
|\{t \in [0,1] \,|\, \lambda_0 \in \sigma_{p}(H_t)\}| = 0.
\end{equation}
\end{conjecture}

Here is an elementary result in this context:

\begin{lemma} \lb{l2.3} 
Let $H_0$ and $V$ be self-adjoint in $\cH$ such that  
$V(H_0 - z_0 I_{\cH})^{-1} \in \cB_{\infty} (\cH)$ for some 
$($and hence for all\,$)$ $z_0 \in \rho(H_0)$. 
Consider a fixed $E_0 \in \bbR \cap \rho(H_0)$ and introduce $H_t = H_0 + t V$, $t \in [0,1]$. Then there exist at most finitely many $t \in [0,1]$ 
such that $E_0 \in \sigma_{p}(H_t)$. 
\end{lemma}
\begin{proof}
Applying the Birman--Schwinger principle (cf., e.g., \cite[\S III.2]{Si71}), 
\begin{equation}
E_0 \in \sigma_{p}(H_t) \, \text{ is equivalent to } \, (-1/t) \in \sigma_{p} (V(H_0 - E_0I_{\cH})^{-1}).  
\end{equation}
Since by hypothesis, $V(H_0 - E_0I_{\cH})^{-1} \in \cB_{\infty} (\cH)$, the nonzero eigenvalues of 
$V(H_0 - E_0I_{\cH})^{-1} \in \cB_{\infty} (\cH)$ are either finite in number, or else, converge to zero (they 
need not be real). Either way, there can only be finitely many $t \in [0,1]$ such that $(-1/t)$ is an 
eigenvalue of $V(H_0 - E_0I_{\cH})^{-1}$. 
\end{proof}

However, since we had to assume $E_0 \in \rho(H_0) \cap \bbR$, this basically renders Lemma \ref{l2.3} irrelevant in connection with Conjecture \ref{c2.2}.  

The next lemma is inspired by a result due to Friedlander \cite{Fr02}, who applied it in connection with his proof of purely absolutely continuous spectrum for a certain class of 2nd 
order elliptic differential operators with periodic coefficients satisfying a symmetry condition. 

\begin{lemma} $($\cite{Fr02}, main theorem and its proof\,$)$ \lb{l2.4}
Let $H_0$ be self-adjoint in $\cH$, and suppose 
that $V^* = V \in \cB(\cH)$ and $V(H_0 + z_0 I_{\cH})^{-1} \in \cB_{\infty}(\cH)$ 
for some $($and hence for all\,$)$ $z_0 \in \rho(H_0)$. In 
addition, assume that $0 \in \rho(H_0) \cup \sigma_d(H_0)$, and \\
$(i)$ \hspace*{.01mm} $\rank(V_-) < \infty$, \\
$(ii)$ $\ker(V) = \{0\}$. \\
Then,
\begin{equation}\lb{2.5aaaa}
\{t \in \bbC \, | \, 0 \in \sigma_p(H_0 + t V)\} \, \text{ is discrete in $\bbC$}
\end{equation} 
$($i.e., a subset of $\bbC$ without finite accumulation point\,$)$. 
\end{lemma} 
\begin{proof} 
(The case $0 \in \rho(H_0)$ is just Lemma \ref{l2.3}, but we will not use this in 
this proof.) We will prove the folklore-type result that for some discrete set 
$\cE\subset \bbC$, and 
for some $N_0\in \bbN\cup \{0\}$, $\dim(\ker(H_0+tV))= N_0$ for all $t \in \bbC \backslash \cE$.
If $N_0=0$, the proof is finished. In the case $N_0\in \bbN$, we will seek to derive a 
contradiction. 

Introducing $P_0 = E_{H_0}(\{0\})$ and 
$P_0^\bot = I_{\cH} - P_0$, then $H_0 u(t) = - t V u(t)$ implies 
$P_0^\bot H_0 P_0^\bot u(t) = - t P_0^\bot V u(t)$, and hence, 
$P_0^\bot u(t) = - t \big(H_0^\bot\big)^{-1} P_0^\bot V u(t)$, where 
$H_0^\bot = P_0^\bot H_0 P_0^\bot$. Thus, assuming $|t| \leq R$ for some 
$R>0$, we decompose, 
\begin{align}
u(t) &= P_0 u(t) + P_0^\bot u(t) = P_0 u(t) 
- t \big(H_0^\bot\big)^{-1} P_0^\bot V u(t)    \no \\
&= P_0 u(t) - t \big[\big(H_0^\bot\big)^{-1} P_0^\bot V P_0 + 
\big(H_0^\bot\big)^{-1} P_0^\bot V P_0^\bot \big] u(t)   \no \\
&= \big[P_0 - t F_0 - t \wti F_n - t G_n\big] u(t)    \no \\
&= \big[P_0 - t F_n - t G_n\big] u(t), 
\end{align}
where 
\begin{align}
& F_0 = \big(H_0^\bot\big)^{-1} P_0^\bot V P_0, \quad \rank(F_0) < \infty,   \\
& \big(H_0^\bot\big)^{-1} P_0^\bot V P_0^\bot 
= \big[\wti F_n + G_n\big] \in \cB_{\infty}(\cH),    \\
& F_n = F_0 + \wti F_n, \quad \rank\big(\wti F_n\big) < \infty, \; 
\rank(F_n) < \infty,    \\
& G_n = G P_n^\bot \, \text{ for } \, G:=(H_0^\bot)^{-1}P_0^\bot VP_0^\bot \in \cB_{\infty}(\cH), 
\quad  \|G_n\|_{\cB(\cH)} \leq (2R)^{-1}, 
\end{align}
and $P_n$ is a strictly monotone increasing sequence of orthogonal projections, $P_n \leq P_{n+1}$, $n \in \bbN$, $\slim_{n \to \infty} P_n = I_{\cH}$, 
$P_n^\bot = I_{\cH} - P_n$, $n \in \bbN$, such that for $n \in \bbN$ 
sufficiently large, 
\begin{equation}
F_n P_n^\bot = 0, \quad 
F_0 P_n^\bot = 0.    \lb{2.11a} 
\end{equation}
Thus,
\begin{equation}
u(t) = [I_{\cH} + t G_n]^{-1} [P_0 - t F_n] u(t), \quad |t| \leq R,   \lb{2.12a}
\end{equation}
and finally, applying $P_n$ to \eqref{2.12a}, employing 
\eqref{2.11a}, the equation
$[H_0 + t V]u(t)=0$, $|t| \leq R$, is reduced to the matrix equation,
\begin{equation}\lb{2.13aaaa}
[P_n - P_n(I_{\cH}+tG_n)^{-1}[P_0-tF_n][P_nu(t)=0]\quad |t| \leq R.
\end{equation}
A solution $P_nu(t)$ to \eqref{2.13aaaa} yields an element $0\neq u(t)\in \ker(H_0+tV)$ 
upon taking the component of $u(t)$ orthogonal to $P_nu(t)$ to be 
$[I_{\cH}-P_n][I_{\cH}+tG_n]^{-1}[P_0-tF_n]P_nu(t)$, assuming without loss that the projections $P_n$ are chosen so that $P_0\leq P_n$, $n\in \bbN$, and hence, 
\begin{equation}
u(t) = P_nu(t)+[I_{\cH}-P_n][I_{\cH}+tG_n]^{-1}[P_0-tF_n]P_nu(t).
\end{equation}

Thus, \eqref{2.13aaaa} either has a positive dimensional subspace of solutions 
$P_n u(t_k)$, $1 \leq k \leq N_R$, in each disk $|t| \leq R$ and $\dim(\ker(H_0+tV))=0$ outside 
the discrete set $\{t_k\}$ in the disc (and \eqref{2.5aaaa} follows in the disk), or else, 
$\det(P_n - P_n(I_{\cH} + tG_n)^{-1}[P_0-tF_n])$ vanishes identically in $|t| \leq R$.  In the latter case, 
constancy of the kernel dimension outside a discrete set follows easily from the analyticity of the minors in $t$, including those of maximal size that are not identically zero. It is enough to notice 
that in the intersection of discs $D_r$ (see below), the  dimension of the space of solutions of 
\eqref{2.13aaaa} should be the same constant for both discs despite of choosing different 
projections $P_n$ associated with them. This follows immediately from the coincidence of the  dimension of the 
solution space for \eqref{2.13aaaa} with the kernel dimension for the operator $H_0 + t V$.
 We intend to assume that the constant kernel dimension, call it $N_0$, is nonzero and then 
 ultimately arrive at a contradiction.

Denoting $D_{r}:=\{t\in \bbC\,|\, |t|\leq r\}$, $r>0$, one may choose a sequence of radii $\{R_j\}_{j=1}^{\infty}$ diverging monotonically to $+\infty$ for which the corresponding matrix equation \eqref{2.13aaaa} has a nontrivial solution $u_j(t)$ for all $|t|<R_j$ apart from a discrete set $t\in\cE_j\subset D_{R_j}$.  By taking $u(t)=u_1(t)$ if $0\leq |t|<R_1$ and $u(t)=u_j(t)$ if $R_j\leq |t|<R_{j+1}$ for some $j\in \bbN$, one concludes that there is a discrete 
set $\cE$ in $\bbC$ such that for $t \in \bbC \backslash \cE$, there exists an $N_0$-dimensional subspace of solutions $u(t) \in \dom(H_0)$, satisfying $H_0 u(t) = - t V u(t)$, $t \in \bbC \backslash \cE$.   If $N_0$  is positive one may choose a family of vectors $u(t)\in \ker(H_0+tV)$ with $\|u(t)\|_{\cH}=1$, $t\in \bbC\backslash \cE$.  Note that this family need not to be continuous in $t$.  (The exceptional points from $\cE$ in the disc $D_{R_j}$ are simply those in the disc $D_{R_{j-1}}$ together with any additional ones in the annulus $D_{R_j}\backslash D_{R_{j-1}}$.  This procedure allows to avoid the possibility of exceptional points accumulating locally as the radii diverge to infinity.)

Let $s, t \in \bbC \backslash \cE$, then 
\begin{equation}\lb{2.15aaaa}
0 = (u(s), (H_0 + t V) u(t))_{\cH} = (t - {\ol s}) (u(s), V u(t))_{\cH},
\end{equation}
that is, 
\begin{equation}\lb{2.16aaaa}
(u(s), V u(t))_{\cH} = 0, \quad s,t \in \bbC \backslash \cE, \; {\ol s} \neq t.
\end{equation}
The identity in \eqref{2.16aaaa} extends to 
\begin{equation}\lb{2.17aaaa}
(u(s), V u(s))_{\cH} = 0, \quad s \in \bbC \backslash \cE. 
\end{equation}
Indeed, \eqref{2.17aaaa} follows immediately from \eqref{2.15aaaa} if $s\in \bbC\backslash \cE$ and $\Im(s)\neq 0$.  If $s\in \bbR\backslash \cE$ is fixed, then in a sufficiently small neighborhood $|t - s|<\varepsilon$ in $\bbC$, it is possible to choose a norm-continuous vector-valued function $v(t)\in \ker(H_0+tV)$ for which $v(s)=u(s)$.  Then, in analogy to \eqref{2.15aaaa}, one obtains
\begin{equation}
0 = (u(s), (H_0 + t V) v(t))_{\cH} = (t - s) (u(s), V v(t))_{\cH},\quad 0<|t-s|<\varepsilon,
\end{equation}
so that
\begin{equation}\lb{2.19aaaa}
(u(s),Vv(t))_{\cH} = 0,\quad 0<|t-s|<\varepsilon.
\end{equation}
Now \eqref{2.17aaaa} follows immediately upon taking the limit $t \to s$ in \eqref{2.19aaaa}.  
From this point on one can follow Friedlander's argument in \cite[p.~54]{Fr02}. Setting
\begin{equation}
\cU = \ol{\linspan \, \{u(t) \in \ker(H_0 + t V) \, | \, t \in \bbC \backslash \cE\}},
\end{equation}
one concludes that
\begin{equation}
(u, V v)_{\cH} = 0, \quad u, v \in \cU, \, \text{ and hence, } \, 
(u, V u)_{\cH} = 0, \quad u \in \cU. 
\end{equation}
Thus,
\begin{equation}
(u,V u)_{\cH} \leq 0, \; u \in \cU, \, \text{ and } \, \rank(V_-) < \infty \, 
\text{ imply } \, \dim(\cU) < \infty.
\end{equation}
Next, let $u(t_k) \in \cU$, $t_k \in \bbC \backslash \cE$, $k \in \bbN$, such that 
$|t_k| \to \infty$ as $k \to \infty$. Then,
\begin{equation}
- t_k^{-1} \big(H \big|_{\cU}\big) u(t_k) = V u(t_k), \quad k \in \bbN,
\end{equation}
together with $\|u(t_k)\|_{\cH} = 1$ and $\dim(\cU) < \infty$ yields the existence 
of a subsequence $\{u(t_{k_j})\}_{j \in \bbN}$ of $\{u(t_k)\}_{k \in \bbN}$ and 
a $u_{\infty} \in \cU$, with $\|u_{\infty}\|_{\cH} = 1$, such that 
\begin{equation}
\lim_{j \to \infty} \|u(t_{k_j}) - u_{\infty}\|_{\cH} = 0 \, \text{ and } \, 
- t_{k_{j}}^{-1} \big(H_0\big|_{\cU}\big) u(t_{k_j}) = V u(t_{k_j}), \quad 
j \in \bbN.    \lb{2.13a} 
\end{equation} 
Since $H_0\big|_{\cU}$ and $V$ are bounded, taking the limit $j \to \infty$ 
in the second equality in \eqref{2.13a} then yields 
\begin{equation}
0 = V u_{\infty}.
\end{equation}
Since $\ker(V) = \{0\}$ by assumption, one finally concludes $u_{\infty} = 0$ and 
hence $\cU = \{0\}$, contradicting $N_0 \in \bbN$. 
\end{proof}

We note that condition $(i)$ in Lemma \ref{l2.4} is necessary because of 
Remark \ref{r2.10}, whereas condition $(ii)$ is necessary because of 
Example \ref{e2.9}. 

We continue with a relevant positive result in the special case where 
$0 \leq V \in \cB(\cH)$, that is derived in the proof of \cite[Lemma\ 4]{NNS13}. 

\begin{lemma} $($\cite{CH94}, \cite[Lemma\ 4 and its proof]{NNS13}$)$ \lb{l2.5} 
Assume Hypothesis \ref{h2.1},  let $H_0$ be self-adjoint in $\cH$, and 
$0 \leq V \in \cB(\cH)$. Consider 
\begin{equation}
H_t = H_0 + t V, \quad \dom(H_t) = \dom(H_0), \quad t \in [0,1], 
\end{equation} 
and assume that $\ran(V)$ is a generating $($i.e., cyclic\,$)$ subspace for $H_0$, that is, 
\begin{equation}
\cH =\ol{\linspan \, \{(H_0-zI_{\cH})^{-1} V e_n\in \cH \, \vert\, n \in \cI, \, z\in \bbC \backslash \bbR\}}     
\end{equation}
$($equivalently, $\cH=\ol{\linspan \, \{E_{H_0}(\lambda) V e_n\in \cH \, \vert \, n \in \cI, \, \lambda \in \bbR\}}$$)$.
Suppose that $\lambda_0 \in \sigma_{p}(H_0)$. Then 
\begin{equation}\lb{2.6}
|\{t \in [0,1] \,|\, \lambda_0 \in \sigma_{p}(H_t)\}| = 0.
\end{equation}
\end{lemma}
\begin{proof}
Fix $\lambda_0\in \bbR$.  Extend $H_t$ from $t\in[0,1]$ to $t\in \bbR$ by
\begin{equation}
H_t = H_0 + t V,\quad t\in \bbR.
\end{equation}
If 
\begin{equation}
E_{H_t}(\{\lambda_0\}) := E_{H_t}(\lambda_0) - E_{H_t}(\lambda_0-0),\quad t\in \bbR,
\end{equation}
then spectral averaging (cf., e.g., \cite[Corollary 4.2 and its proof]{CH94}) immediately yields
\begin{equation}\lb{2.7}
\int_{\bbR} \frac{(v,V^{1/2}E_{H_t}(\{\lambda_0\})V^{1/2}v)_{\cH}}{1+t^2}dt=0,\quad v\in \cH.
\end{equation}
Since $V^{1/2}E_{H_t}(\{\lambda_0\})V^{1/2}\geq 0$, $t\in \bbR$, \eqref{2.7} implies
\begin{equation}\lb{2.8}
V^{1/2}E_{H_t}(\{\lambda_0\})V^{1/2}v=0,\quad t\in \bbR\backslash N_v, \; v\in \cH,
\end{equation}
where $N_v\subseteq \bbR$ is (Lebesgue) measurable with $|N_v|=0$, $v\in \cH$.  Applying \eqref{2.8} to each element of a complete orthonormal system for $\cH$, we obtain a (Lebesgue) measurable set $N\subseteq \bbR$ with $|N|=0$ for which
\begin{equation}\lb{2.9}
V^{1/2}E_{H_t}(\{\lambda_0\})V^{1/2}=0,\quad t\in \bbR\backslash N.
\end{equation}
Now, \eqref{2.9} implies $E_{H_t}(\{\lambda_0\})|_{\ran(V)}=0$, $t\in \bbR\backslash N$.  Indeed, by \eqref{2.9}, 
\begin{align}
\big\|E_{H_t}(\{\lambda_0\})Vv \big\|_{\cH}^2&= (V v,E_{H_t}(\{\lambda_0\})Vv)_{\cH}\no\\
&= \big(V^{1/2}v, [V^{1/2}E_{H_t}(\{\lambda_0\})V^{1/2}]V^{1/2}v\big)_{\cH}=0,\lb{2.10}\\
&\hspace*{4.5cm} t\in \bbR\backslash N,\; v\in \cH.\no
\end{align}
Next one notes, 
\begin{equation}\lb{2.11}
\begin{split}
E_{H_t}(\{\lambda_0\})(H_t-zI_{\cH})^{-1}Ve_n= (H_t-zI_{\cH})^{-1}E_{H_t}(\{\lambda_0\})Ve_n=0,&\\
\quad n\in \cI,\; z\in \bbC\backslash \bbR,\; t\in \bbR\backslash N,&
\end{split}
\end{equation}
where $\{e_n\}_{n\in \cI}$, $\cI\subseteq \bbN$ an appropriate index set, is any complete orthonormal system in $\cH$.  By a standard resolvent identity, one obtains $t$-invariance of the cyclic subspace for $H_t$ generated by $\ran(V)$ in the form:
\begin{equation}\lb{2.12}
\begin{split}
&\ol{\linspan \, \{(H_0-zI_{\cH})^{-1} V e_n\, \vert\, n \in \cI, \, z\in \bbC \backslash \bbR\}}\\
&\quad =\ol{\linspan \, \{(H_t-zI_{\cH})^{-1} V e_n\, \vert\, n \in \cI, \, z\in \bbC \backslash \bbR\}},\quad t\in \bbR.
\end{split}
\end{equation}
Since the first subspace in \eqref{2.12} coincides with $\cH$, \eqref{2.11} and boundedness of the spectral projections $E_{H_t}(\{\lambda_0\})$ imply
\begin{equation}\lb{2.13}
E_{H_t}(\{\lambda_0\})u=0,\quad u\in \cH,\; t\in \bbR\backslash N.
\end{equation}
Since $N$ is a set with zero Lebesgue measure, $\lambda_0\notin \sigma_{p}(H_t)$ for 
a.e.~$t\in \bbR$.  The result in \eqref{2.6} follows immediately.
\end{proof}

Next, we continue this line of thought and relate it to spectral theory for linear pencils of operators. 
Considering the standard decomposition of $V$
\begin{equation}
V = V_+ - V_-, \quad 0 \leq V_{\pm} = [|V| \pm V]/2 \in \cB(\cH),      \lb{3.54}
\end{equation}

\noindent
provided by the spectral theorem for $V$, we now modify this to a more general 
(highly nonunique) decomposition 
\begin{equation}\lb{2.18}
V = V_1 - V_2, \quad 0 \leq V_j \in \cB(\cH). 
\end{equation}
In particular, we may assume that either 
\begin{equation}
V_1 \geq \varepsilon I_{\cH}, 
\end{equation}
or, alternatively, 
\begin{equation}
V_2 \geq \varepsilon I_{\cH}, 
\end{equation}
(e.g., by adding an appropriate multiple of the identity to $V_1$ or $V_2$). In this case, the 
basic eigenvalue equation 
\begin{equation}
H_t \psi = \lambda_0 \psi, \quad \psi \in \cH, \; \lambda_0 \in \bbR, \; t \in [0,1],    \lb{3.58} 
\end{equation}
is equivalent to 
\begin{align}
\begin{split} 
\big[V_1^{-1/2} H_0 V_1^{-1/2} - \lambda_0 V_1^{-1}\big] \big(V_1^{1/2} \psi\big) 
= - t \big[I_{\cH} - V_1^{-1/2} V_2 V_1^{-1/2} \big] \big(V_1^{1/2} \psi\big)&      \lb{3.59} \\ 
\text{ if } \, V_1^{-1} \in \cB(\cH),& 
\end{split} 
\end{align}
or, alternatively, to 
\begin{align}
\begin{split} 
\big[V_2^{-1/2} H_0 V_2^{-1/2} - \lambda_0 V_2^{-1}\big] \big(V_2^{1/2} \psi\big) 
= t \big[I_{\cH} - V_2^{-1/2} V_1 V_2^{-1/2}\big] \big(V_2^{1/2} \psi\big)&     \lb{3.60} \\ 
\text{ if } \, V_2^{-1} \in \cB(\cH).& 
\end{split}  
\end{align}
Thus, the standard self-adjoint eigenvalue problem \eqref{3.58} is equivalent to a 
linear pencil eigenvalue problem (w.r.t. the parameter $t \in [0,1]$) of the form
\begin{equation}
A(\lambda_0) f = t B f, \quad f \in \cH, \; t \in [0,1],       \lb{3.61} 
\end{equation}
where 
\begin{equation}
A(\lambda_0) = A(\lambda_0)^* \in \cB(\cH), \quad B = B^* \in \cB(\cH).
\end{equation}
That is, the underlying linear self-adjoint pencil is of the form $A(\lambda_0) - t B$, $t\in [0,1]$.

Generally speaking, the pencil $A(\lambda_0) - t B$, $t\in [0,1]$, readily leads to non-real eigenvalues (i.e., becomes an inequivalent non-self-adjoint spectral problem) even if $A(\lambda_0)$ as well as $B$ are self-adjoint, as the following elementary example illustrates.

\begin{example} \cite[\S VI.1.1, Example 1.14]{SS90} \lb{e2.6} 
In the Hilbert space $\bbC^2$, the following linear pencil
\begin{equation}
\begin{pmatrix}
1 & 0\\
0 & -1
\end{pmatrix}
f = t 
\begin{pmatrix}
0 & 1\\
1 & 0
\end{pmatrix}
\end{equation}
has eigenvalues $\pm i$ even though both matrix coefficients are self-adjoint.
\end{example}

While the eigenvalues of a standard self-adjoint eigenvalue problem in a separable Hilbert space are necessarily countable, no such result holds for pencil eigenvalue problems, as demonstrated by the following example kindly communicated to us by T.\ Azizov \cite{Az13}. 

\begin{example} [\cite{Az13}] \lb{e2.7} 
In the Hilbert space $\cH$ consider the complete orthonormal basis $\{e_k\}_{k \in \bbN}$ and introduce the shift operator $S_+$
\begin{equation}
S_+ e_k = e_{k+1}, \quad k \in \bbN, 
\end{equation}  
such that 
\begin{equation}
S_+^* e_1 = 0, \quad S_+^* e_k = e_{k-1}, \quad k \in \bbN, \; k \geq 2. 
\end{equation}  
Then, with $\bbD = \{z \in \bbC \,|\, |z| < 1\}$, the open unit disk in $\bbC$, one has the following facts 
$($cf., e.g., \cite[p.\ 468--469]{Ku01}$)$,
\begin{align}
& \sigma_{p} (S_+) = \sigma_{r} (S_+^*) = \emptyset,    \lb{3.65} \\
& \sigma_{r} (S_+) = \sigma_{p} (S_+^*) = \bbD,    \lb{3.66} \\
& \sigma_{c} (S_+) = \sigma_{c} (S_+^*) = \partial \bbD,    \lb{3.67} \\
& \sigma (S_+) = \sigma (S_+^*) = \ol \bbD.       \lb{3.68} 
\end{align} 
Next, introduce in $\cK = \cH \oplus \cH$ the self-adjoint $2 \times 2$ block operator matrices 
\begin{equation}
H_0 = \begin{pmatrix} 0 & S_+ \\ S_+^* & 0 \end{pmatrix}, \quad 
V = \begin{pmatrix} 0 & I_{\cH} \\ I_{\cH} & 0 \end{pmatrix}.     \lb{3.69} 
\end{equation}
Then $V^2 = I_{\cK} = \begin{pmatrix} I_{\cH} & 0 \\ 0 & I_{\cH} \end{pmatrix}$ shows that 
the linear pencil eigenvalue problem
\begin{equation}
H_0 f = t V f, \quad f \in \cK, \; t \in \bbC,      \lb{3.70} 
\end{equation}
is equivalent to the standard $($non-self-adjoint\,$)$ eigenvalue problem
\begin{equation}
V H_0 f = t f, \quad f \in \cK, \; t \in \bbC, \quad 
V H_0 = \begin{pmatrix} S_+^* & 0 \\ 0 & S_+ \end{pmatrix}.     \lb{3.71} 
\end{equation}
Together with \eqref{3.65}--\eqref{3.68} this yields
\begin{align}
\sigma (V H_0) = \ol \bbD, \quad \sigma_{p} (V H_0) = \bbD,     \lb{3.72} 
\end{align}
in particular, each $t \in \bbD$ is an eigenvalue for \eqref{3.70}. 
\end{example}

\begin{remark} \lb{r2.8} 
Further generalizations of Example \ref{e2.7} are possible: \\[.6mm]
$(i)$ In Example \ref{e2.7} the operator $V$ is invertible and therefore not compact. This may unintentionally lead to the  misunderstanding that the eigenvalue phenomenon is related to the noncompactness of $V$. However, a suitable reconstruction of Example \ref{e2.7}, replacing the identity operators in $V$ (see \eqref{3.69}) by a compact self-adjoint diagonal operator $\Lambda _2$ (in the basis  $\{e_k\}_{k \in \bbN}$) and the operators $S_+$ (resp., $S_+^{*}$) in $ H_0$ by 
$S_+\Lambda_1$ (resp., $\Lambda_1$$S_+^{*}$), settles this issue. Here $\Lambda_1$ is again a self-adjoint diagonal operator in  $\cH$. A suitable choice of the diagonal operators  leads to an example where the operator $V$ is compact, but the set of eigenvalues $ { t  }$ of the spectral problem  \eqref{3.70} covers the whole complex plane. Surely, the critical condition that the range of $V$ is a generating subspace for $H_0$ can be satisfied. \\[.6mm]
$(ii)$ Another generalization is related to the case where the positive part of the self-adjoint operator $V$ is invertible (on a suitable subspace of  $\cH$) and its negative part is a compact operator.  In that case, as explicit examples show, the point spectrum of the spectral problem   \eqref{3.70} may cover the whole disc $ \bbD$  again. We note that the vanishing of the negative part according to the Lemma \ref{l2.5} leads to the fact that the set of eigenvalues has Lebesgue measure 0.
\end{remark}

Next, to put Conjecture \ref{c2.2} to rest once and for all, we offer the following elementary three-dimensional counterexample:

\begin{example} \lb{e2.9} 
Consider $\cH = \bbC^3$, 
\begin{equation}
H_0 = \begin{pmatrix} 1 & 1 & 1 \\ 1 & 1 & 1 \\ 1 & 1 & 0  \end{pmatrix} = H_0^*, \quad 
V = \begin{pmatrix} 1 & 0 & 0 \\ 0 & -1 & 0 \\ 0 & 0 & 0  \end{pmatrix} = V^*. 
\end{equation}
Then,
\begin{equation}
\ker(H_0) = {\rm lin.span} \begin{pmatrix} 1 \\ -1 \\ 0 \end{pmatrix}, 
\quad \ran(V) = \bbC^2 \oplus \{0\},
\end{equation}
and since 
\begin{equation}
H_0 \begin{pmatrix} 1 \\ 0 \\ 0 \end{pmatrix} = \begin{pmatrix} 1 \\ 1 \\ 1 \end{pmatrix},
\end{equation}
one infers that 
\begin{equation}
\ran(V) \dotplus H_0 \ran(V) = \bbC^3 
\end{equation}
$($with $\dotplus$ denoting the direct, though, not necessarily orthogonal direct 
sum, which we denote by $\oplus$\,$)$. Finally, introducing 
\begin{equation}
f_t = \begin{pmatrix} 1 \\ -1 \\ -t \end{pmatrix}, \quad t \in \bbC, 
\end{equation}
one obtains
\begin{equation}\lb{2.41}
H_0 f_0 = 0, \quad (H_0 + t V) f_t = 0 \, \text{ for each $t \in \bbC$,}
\end{equation}
illustrating in dramatic fashion that nonnegativity of $V$ cannot be omitted in Lemma \ref{l2.5}.

To cast \eqref{2.41} as an equivalent problem for a linear pencil of operators, decompose $V$ according to $V = V_1 - V_2$ with
\begin{equation}
V_1 = I_3 = \begin{pmatrix}
1 & 0 & 0\\
0 & 1 & 0\\
0 & 0 & 1
\end{pmatrix},
\quad 
V_2 =\begin{pmatrix}
0 & 0 & 0\\
0& 2 & 0\\
0 & 0 & 1
\end{pmatrix},
\end{equation}
which yields a decomposition of $V$ of the form \eqref{2.18}.  By \eqref{3.59}, 
the eigenvalue problem \eqref{2.41} is equivalent to the linear pencil eigenvalue problem
\begin{equation}
H_0f = -t(I_3 - V_2)f,\quad f\in \bbC^3,\, t\in \bbC.
\end{equation}
\end{example}

\begin{remark} \lb{r2.10} 
In Example \ref{e2.9}, both matrices $H_0$ and $V$ fail to be invertible.  Surely in a finite dimensional space $\cH$, the invertibility of at least one of $H_0$ or $V$ immediately leads to the reduction of the spectral problem  \eqref{3.70} to the standard eigenvalue problem in the spectral parameter $t$ or $ t^{-1}$ and therefore to the finiteness of the set of $t$ values. 
\end{remark}


\medskip

{\bf Acknowledgments.} We are indebted to Fedor Sukochev for raising a question that naturally led to what we called Conjecture \ref{c2.2}. Helpful discussions with Tomas Azizov, Rostyslav Hryniv, Vladimir Derkach, Heinz Langer, Michael Levitin, Mark Malamud, Marco Marletta, Leonid Parnovski, 
Fritz Philipp, Andrei Shkalikov, Barry Simon, G\"unter Stolz, Vadim Tkachenko, and Christiane Tretter are gratefully acknowledged. Especially, we thank Tomas Azizov for sending us Example \ref{e2.7} and Leonid Parnovski for pointing out Friedlander's paper \cite{Fr02} to us.  

S.\,N. is grateful to the Department of Mathematics of the University of Missouri where part of this work was completed while on a Miller Scholar Fellowship in February--March of 2014.


\end{document}